\pdfoutput=1 

\RequirePackage[l2tabu, orthodox]{nag}

\documentclass[12pt]{amsart}
\usepackage[utf8]{inputenc} 
\usepackage{fullpage,url,amssymb,enumerate,colonequals,graphicx}
\usepackage[all]{xy} 
\usepackage{mathrsfs} 

    \DeclareFontFamily{U}{wncy}{}
    \DeclareFontShape{U}{wncy}{m}{n}{<->wncyr10}{}
    \DeclareSymbolFont{mcy}{U}{wncy}{m}{n}
    \DeclareMathSymbol{\Sha}{\mathord}{mcy}{"58} 

\usepackage[T1]{fontenc}
\usepackage[utf8]{inputenc}

\usepackage[dvipsnames,xcdraw,hyperref]{xcolor}

\newcommand{\defi}[1]{\textsf{#1}} 

\newcommand{\Aff}{\mathbb{A}}
\newcommand{\C}{\mathbb{C}}

\newcommand{\PP}{\mathbb{P}}
\newcommand{\Q}{\mathbb{Q}}
\newcommand{\R}{\mathbb{R}}
\newcommand{\Z}{\mathbb{Z}}

\newcommand{\pp}{\mathfrak{p}}

\newcommand{\calA}{\mathcal{A}}

\newcommand{\calO}{\mathcal{O}}

\DeclareMathOperator{\Aut}{Aut}

\DeclareMathOperator{\den}{den}

\DeclareMathOperator{\Gal}{Gal}

\DeclareMathOperator{\Lie}{Lie}
\DeclareMathOperator{\N}{N}

\DeclareMathOperator{\num}{num}

\DeclareMathOperator{\rank}{rank}

\newcommand{\injects}{\hookrightarrow}
\newcommand{\Intersection}{\bigcap} 

\newcommand{\To}{\longrightarrow}
\newcommand{\union}{\cup} 
\newcommand{\Union}{\bigcup} 

\newtheorem{theorem}{Theorem}[section]
\newtheorem{lemma}[theorem]{Lemma}

\theoremstyle{definition}

\newtheorem{example}[theorem]{Example}

\theoremstyle{remark}
\newtheorem{remark}[theorem]{Remark}

\makeatletter
\g@addto@macro\bfseries{\boldmath} 
\makeatother

\usepackage{tikz}
\usepackage{microtype}  

\usepackage[
	pagebackref,
	pdfauthor={Bjorn Poonen}, 
]{hyperref}

\begin{document}

\title{Rank stability makes rings of integers diophantine}
\author{Bjorn Poonen}
\thanks{This research was supported in part by National Science Foundation grant DMS-2101040 and Simons Foundation grants \#402472 and \#550033.}
\address{Department of Mathematics, Massachusetts Institute of Technology, Cambridge, MA 02139-4307, USA}
\email{poonen@math.mit.edu}
\urladdr{\url{http://math.mit.edu/~poonen/}}
\date{October 22, 2025}

\begin{abstract}
The recent negative answer to Hilbert's tenth problem over rings of integers relies on a theorem that for every extension of number fields $L/K$, if there is an abelian variety $A$ over $K$ such that $0 < \rank A(K) = \rank A(L)$, then $\calO_K$ is $\calO_L$-diophantine.
We present an alternative proof of this theorem and review how it is used.
\end{abstract}

\maketitle

\section{Introduction}\label{S:introduction}

\subsection{History}\label{S:history}

Hilbert's tenth problem asked for an algorithm to decide, given a multivariable polynomial equation with integer coefficients, whether it has a solution in integers.
By \cite{Davis-Putnam-Robinson1961,Matiyasevich1970}, there is no such algorithm.

For each number field $K$, replacing $\Z$ with the ring of integers $\calO_K$ yields a new question.
The negative answer for $\Z$ implies a negative answer for $\calO_K$ \emph{if $\Z$ is $\calO_K$-diophantine}; see Section~\ref{S:diophantine} for the definition.
This led Denef and Lipshitz \cite{Denef-Lipshitz1978} to conjecture that $\Z$ is $\calO_K$-diophantine for \emph{every} number field $K$.
Their conjecture was proved for many classes of number fields by using the structure of integer points on algebraic tori, specifically, Pell equations \cite{Denef1975,Denef-Lipshitz1978,Denef1980,Pheidas1988,Shlapentokh1989}.

Starting with Denef, various authors \cite{Denef1980,Poonen2002,Cornelissen-Pheidas-Zahidi2005,Shlapentokh2008,Mazur-Rubin-Shlapentokh2024-existential} showed that one could use elliptic curves or abelian varieties in place of algebraic tori, if certain rank conditions could be proven.
The strongest of these results states, for an extension $L/K$ of number fields, that if the condition
\[
\calA_{K,L}\colon \textup{There exists an abelian variety $A$ over $K$ such that $0 < \rank A(K) = \rank A(L)$.}
\]
holds, then $\calO_K$ is $\calO_L$-diophantine \cite[Theorem~1.1]{Mazur-Rubin-Shlapentokh2024-existential} (in fact, this result applies to some \emph{infinite} algebraic extensions as well).
Via such results, \cite{Mazur-Rubin2010,Mazur-Rubin2018,Murty-Pasten2018,GarciaFritz-Pasten2020,Pasten2023,Shnidman-Weiss2023,Kundu-Lei-Sprung2024,Ray-Weston2024} proved that $\Z$ is $\calO_F$-diophantine for many new number fields $F$.

Recently, \cite{Koymans-Pagano2025+}, using an input from additive combinatorics, constructed elliptic curves proving $\calA_{K,L}$ for enough degree~$2$ extensions $L/K$ to prove the full Denef--Lipshitz conjecture.
Soon thereafter, \cite[Theorem~1.1]{Alpoge-Bhargava-Ho-Shnidman2025+} proved $\calA_{K,L}$ for all degree~$2$ extensions $L/K$, constructing abelian varieties that were not necessarily elliptic curves.
In contrast with \cite{Koymans-Pagano2025+}, \cite{Alpoge-Bhargava-Ho-Shnidman2025+} requires no additive combinatorics beyond a number field analogue of Vinogradov's method from the 1930s.
Later, \cite[Theorem~1.2]{Zywina2025+} proved the stronger theorem that for every degree~$2$ extension $L/K$, there exist infinitely many elliptic curves $E$ over $K$ with $\rank E(K) = \rank E(L) = 1$.

\subsection{Outline}\label{S:outline}

The proof that $\Z$ is $\calO_F$-diophantine for all number fields $F$ can be broken into four independent steps:
\begin{enumerate}[\upshape (i)]
\item \label{I:totally real}
      If $E$ is a totally real number field, then $\Z$ is $\calO_E$-diophantine \cite{Denef1980}.
\item \label{I:Koymans-Pagano}
      $\calA_{K,L}$ holds for all degree~$2$ extensions $L/K$ \cite{Koymans-Pagano2025+,Alpoge-Bhargava-Ho-Shnidman2025+,Zywina2025+}.
\item \label{I:AKL implies diophantine}
$\calA_{K,L}$ implies that $\calO_K$ is $\calO_L$-diophantine \cite[Theorem~1.1]{Mazur-Rubin-Shlapentokh2024-existential}.
\item \label{I:Shlapentokh reduction}
      If $\Z$ is $\calO_E$-diophantine for all totally real $E$, and $\calO_K$ is $\calO_L$-diophantine for all degree~$2$ extensions $L/K$, then $\Z$ is $\calO_F$-diophantine for all number fields $F$.
      This reduction is due to Shlapentokh \cite[Theorem~4.8]{Mazur-Rubin-Shlapentokh2024-existential}.
\end{enumerate}

\begin{remark}
What is proved towards \eqref{I:Koymans-Pagano} determines how strong a version of \eqref{I:AKL implies diophantine} is needed.
Specifically, \cite{Koymans-Pagano2025+} constructs elliptic curves and hence needs only \cite{Shlapentokh2008}, whereas \cite{Alpoge-Bhargava-Ho-Shnidman2025+} needs the full abelian variety statement of \cite{Mazur-Rubin-Shlapentokh2024-existential}, and \cite{Zywina2025+} needs only \cite{Poonen2002}.
\end{remark}

We have nothing new to say about \eqref{I:totally real} and \eqref{I:Koymans-Pagano}.
The main purpose of this note is to give an alternative proof of \eqref{I:AKL implies diophantine}; see Theorem~\ref{T:AKL implies diophantine}\eqref{I:OK is OL-diophantine}.
The key ideas are present in the earlier works, but we introduce several simplifications.
In Section~\ref{S:Shlapentokh}, we reproduce Shlapentokh's reduction argument \eqref{I:Shlapentokh reduction} since it is short.

\subsection{Notation}

Let $K$ be a number field.
Let $\calO_K$ be its ring of integers.
Given $a \in K^\times$, there exist unique coprime ideals $I,J \subset \calO_K$ such that $(a)=I/J$; define $\num(a) \colonequals I$ and $\den(a) \colonequals J$.
If $L$ is a finite extension of $K$, and $\alpha,\beta \in \calO_L$ and $I$ is a nonzero ideal of $\calO_K$, the notation $\alpha \equiv \beta \pmod{I}$ means $I \calO_L | \num(\alpha-\beta)$.

\section{Diophantine sets}
\label{S:diophantine}

Let $R = \calO_K$ for some $K$.
For a finite-type $R$-scheme $X$, a subset $S \subset X(R)$ is \defi{$R$-diophantine} if it is $f(Y(R))$ for some finite-type morphism $Y \to X$.
It is not hard to show that a subset $S \subset R = \Aff^1(R)$ is $R$-diophantine if and only if $S = \{a \in R: (\exists x \in R^n) \; g(a,x)=0\}$ for some $g \in R[t,x_1,\ldots,x_n]$.
Finite unions of $R$-diophantine subsets are $R$-diophantine.
Any morphism of finite-type $R$-schemes $X \to X'$ maps $R$-diophantine subsets of $X(R)$ to $R$-diophantine subsets of $X'(R)$.
Applying this to $\Aff^1 \times \Aff^1 \stackrel{\textup{sum}}\To \Aff^1$ shows that if $S,T \subset R = \Aff^1(R)$ are $R$-diophantine, then so is $S+T \colonequals \{s+t: s\in S,\, t \in T\}$.

\begin{lemma}[\protect{\cite[Proposition~1(b)]{Denef-Lipshitz1978}}]
\label{L:nonzero}
The set $\calO_K-\{0\}$ is $\calO_K$-diophantine.
\end{lemma}

\begin{proof}
For any \emph{nonzero} ideal $I \subset \calO_K$, there exists $x \in \calO_K$ with $(2x-1)(3x-1) \equiv 0 \pmod{I}$ (use the Chinese remainder theorem to reduce to the case where $I$ is a power of a prime ideal).
For $a \in \calO_K$, taking $I=(a)$ shows that
\[
	a \ne 0 \quad\iff\quad (\exists x,y \in \calO_K) \; (2x-1)(3x-1) = ya.\qedhere
\]
\end{proof}

Elements of $K$ can be represented in the usual way as equivalence classes $a/b$ of pairs $(a,b)$ with $a,b \in \calO_K$ and $b \ne 0$.
Subsets of $K$ can then be identified with certain subsets of $\calO_K^2$.
Lemma~\ref{L:nonzero} lets us
\begin{itemize}
\item use polynomial equations involving $K$-valued variables in diophantine definitions and 
\item
extend the $\calO_K$-diophantine notion to subsets of $X(K)$ for finite-type $K$-schemes $X$.
\end{itemize}
For a finite extension $L/K$,
\begin{itemize}
\item
choosing a finite presentation of $\calO_L$ as an $\calO_K$-module lets us use polynomial equations involving $\calO_L$-valued variables in constructing $\calO_K$-diophantine sets, and 
\item
if $\calO_K$ is $\calO_L$-diophantine, then any $\calO_K$-diophantine set is $\calO_L$-diophantine.
\end{itemize}

Each ideal $I \subset \calO_K$ can be represented as $(i_1,i_2)$ for some $i_1,i_2 \in \calO_K$.
Conditions involving ideals can be expressed in diophantine terms in terms of the generators.
For example:
\begin{itemize}
\item $a \in I \iff (\exists x,y \in \calO_K)\;\; a=x i_1 + y i_2$;
\item $J|I \iff i_1,i_2 \in J$;
\item $I=J \iff I|J \textup{ and } J|I$; 
\item $I,J$ are coprime $\iff (\exists i \in I)(\exists j \in J) \;\; i+j=1$; 
\item for $s=a/b \in K$, we have $(s)=I/J \iff aJ=bI$; 
\item $I = \num(s) \iff (\exists J) \;\; (s)=I/J \textup{ and $I,J$ are coprime}$; and 
\item $a \equiv b \pmod{I} \iff a-b \in I$.
\end{itemize}

\section{Using congruences}

Here is an example of how to use a congruence to force an algebraic integer to belong to a smaller ring of integers:

\begin{example}
If $\alpha \in \Z[i]$ is such that $|\alpha| < 5$ and $\alpha \equiv k \pmod{10 \Z[i]}$ for some $k \in \Z$, then $\alpha \in \Z$.
One proof: $\bar{\alpha}-\alpha \in 10 \Z[i]$, but $|\bar{\alpha}-\alpha| < 5+5$, so $\bar{\alpha}=\alpha$; that is, $\alpha \in \Z$.  (The dots below are the $\alpha \in \Z[i]$ congruent to an integer modulo $10\Z[i]$.)
\end{example}

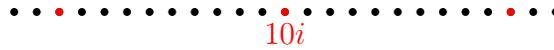
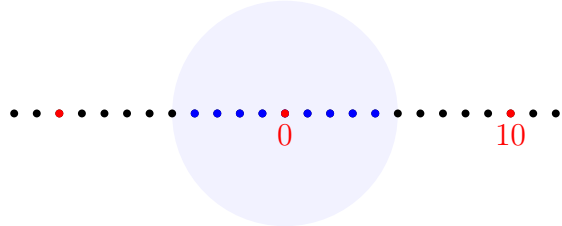
\begin{center}
\begin{tikzpicture}[scale=0.3]


\colorlet{lightblue}{blue!5!white}
\fill[lightblue] (0,0) circle(5);

\foreach \n in {-12,...,12} {
  \foreach \m in {-1,...,1} {
    \fill (\n,10*\m) circle (5pt);
  }
}

\foreach \n in {-4,...,4} {
    \fill[blue] (\n,0) circle (5pt);
}

    \node[red] at (0, 0) [below] {0}; 
    \node[red] at (10, 0) [below] {10}; 
    \node[red] at (0, 10) [below] {$10i$}; 

\fill[red] (-10,10) circle (5pt);
\fill[red] (0,10) circle (5pt);
\fill[red] (10,10) circle (5pt);

\fill[red] (-10,0) circle (5pt);
\fill[red] (0,0) circle (5pt);
\fill[red] (10,0) circle (5pt);

\fill[red] (-10,-10) circle (5pt);
\fill[red] (0,-10) circle (5pt);
\fill[red] (10,-10) circle (5pt);

\end{tikzpicture}
\end{center}

To generalize to $\calO_L \supset \calO_K$ in place of $\Z[i] \supset \Z$, and an $\calO_K$-ideal $I$ in place of $10 \Z[i]$, we use a condition $(\alpha-1) \cdots (\alpha-n) \mid I$ to express that ``$I$ is much bigger than $\alpha$'':

\begin{lemma}
\label{L:congruence}
Fix number fields $L \supset K$.
There exists $n \ge 1$ such that for all $\alpha \in \calO_L$, all nonzero ideals $I \subset \calO_K$, and all $k \in K$,
\[
	(\alpha-1) \cdots (\alpha-n) \mid I \quad \textup{and} \quad \alpha \equiv k \pmod{I} \quad \implies \quad \alpha \in \calO_K.
\]
\end{lemma}

\begin{proof}
Enlarge $L$ to assume that $L/\Q$ is Galois.
Let $\ell=[L:\Q]$.
Choose $n$ such that $n > 23\ell$ and $10^{n-2\ell} > (4n)^\ell$.
Below, $j$ ranges over integers in $[1,n]$, and $\tau$ ranges over elements of $\Gal(L/\Q)$.
Given $\alpha \in \calO_L$, embed $L$ in $\C$ so that $|\alpha| \ge \lvert {}^\tau \alpha \rvert$ for all $\tau$.
Let $M=|\alpha|$.

Suppose that $(\alpha-1) \cdots (\alpha-n) \mid I$ and $\alpha \equiv k \pmod{I}$, but $\alpha \notin \calO_K$.
Choose $\sigma \in \Gal(L/K)$ with ${}^\sigma \alpha \ne \alpha$.
Applying $\sigma$ to $\alpha \equiv k \pmod{I}$ and subtracting gives
$I \mid ({}^\sigma \alpha - \alpha)$, so
$(\alpha-1)\cdots(\alpha-n) \mid ({}^\sigma \alpha - \alpha)$.
Apply the norm $\N \colon L \to \Q$ and then $|\;|$:
\[
    \prod_{j,\tau} \lvert {}^\tau \alpha - j \rvert
        \le \prod_\tau \lvert {}^{\tau\sigma} \alpha - {}^\tau \alpha \rvert 
        \le (2M)^\ell \quad\textup{(since $\lvert {}^\tau \alpha \rvert \le |\alpha| = M$ for all $\tau$).}
\]

We will contradict the last line by proving that many terms on the left are large.
For each $\tau$, let $J_\tau = \{j : \lvert {}^\tau \alpha - j \rvert < 10 \}$,
so $\#J_\tau \le 20$.
Let $J_0 \colonequals \Union_\tau J_\tau$, so $\#J_0 \le 20\ell$.
Let $J_1 \colonequals \{1,\ldots,n\} - J_0$, so $\#J_1 \ge n-20\ell$.
\begin{itemize}
\item If $M \ge 2n$ (so in particular $M \ge 4$, so $(M/2)^3 \ge 2M$), then $\prod_{j \in J_1} \lvert \alpha-j \rvert \ge (M/2)^{\#J_1} \ge (M/2)^{n-20\ell} > (M/2)^{3\ell} \ge (2M)^\ell$.
If $M < 2n$, then $\prod_{j \in J_1} \lvert \alpha-j \rvert \ge 10^{\#J_1} \ge 10^{n-20\ell} > (4n)^\ell > (2M)^\ell$.
\item For $\tau \ne 1$, we have $\prod_{j \in J_1} \lvert {}^\tau \alpha-j \rvert \ge \prod_{j \in J_1} 10 \ge 1$.
\item For $j \in J_0$, we use $\prod_\tau \lvert {}^\tau \alpha-j \rvert = \lvert \N(\alpha-j) \rvert \ge 1$.
\end{itemize}
Multiplying these inequalities gives $\prod_{j,\tau} \lvert {}^\tau \alpha - j \rvert > (2M)^\ell$, a contradiction.
\end{proof}

\section{Weakly approximating \texorpdfstring{$\Z$}{Z}} \label{S:weak}

Let $K$ be a number field.
For each prime ideal $\pp \subset \calO_K$, let $K_\pp$ be the completion.
Let $S \subset K$.
Say that $S$ \defi{weakly approximates $\Z$} if any of the following equivalent conditions holds:
\begin{enumerate}[\upshape (i)]
\item  $\Z$ is contained in the closure of $S$ in $\prod_\pp K_\pp$.
\item for every $k \in \Z$ and primes $\pp_1,\ldots,\pp_m$ of $\calO_K$, there is a sequence in $S$ converging to $k$ in $K_{\pp_i}$ simultaneously for every $i$;
\item for every $k \in \Z$ and nonzero ideal $I \subset \calO_K$, the congruence $x \equiv k \pmod{I}$ has a solution in $S$.
\end{enumerate}

\begin{lemma}
\label{L:numerator}
If $S \subset K$ weakly approximates $\Z$ and $0 \ne \beta \in \calO_K$, then there exists $s \in S$ with $\beta \mid \num(s)$.
\end{lemma}

\begin{proof}
The congruence $x \equiv 0 \pmod{(\beta)}$ has a solution in $S$.
\end{proof}

\begin{theorem} 
\label{T:AKL implies diophantine}
For an extension of number fields $L/K$, if $\calA_{K,L}$ holds, then
\begin{enumerate}[\upshape (a)]
\item there exists an infinite $\calO_L$-diophantine subset $T \subset K$; and 
\item there exists an $\calO_L$-diophantine subset $S \subset K$ that weakly approximates $\Z$;
\item there exists an $\calO_L$-diophantine subset $U$ with $\Z \subset U \subset \calO_K$;
\item \label{I:OK is OL-diophantine}
the subset $\calO_K$ is $\calO_L$-diophantine.
\end{enumerate}
\end{theorem}

\begin{proof}
Fix $A$ as in $\calA_{K,L}$.
Let $r=(A(L):A(K))$.
Then $A(K)$ is a finite union of cosets of $r A(L)$, so $A(K)$ is $\calO_L$-diophantine.
\begin{enumerate}[\upshape (a)]
\item 
Choose a closed immersion $A \injects \PP^N_K$ for some $N$,
and let $T$ be the set of ratios of projective coordinates
of the points in $A(K) \subset \PP^N(K)$,
excluding ratios with denominator $0$.
Since $A(K)$ is infinite, $T$ is infinite.
By definition, $T$ is $\calO_L$-diophantine.
\item
Let $y_1,\ldots,y_g \in K(A)$ be local parameters for $A$ at $0$.
Define
\[
	S = \left\{ \frac{y(Q)}{y(P)} : P,Q \in A(K), \; y \in \sum_{i=1}^g T y_i \right\} \quad \subset K;
\]
we exclude ratios $y(Q)/y(P)$ in which $y(P)$ or $y(Q)$ is undefined or in which $y(P)=0$.
By definition, $S$ is $\calO_L$-diophantine.

Let $k \in \Z$.
Let $p$ be a prime number.
Let $\pp$ be a prime of $\calO_K$ above $p$.
If $R \to 0$ along a smooth analytic arc in the $p$-adic manifold $A(K_\pp)$ and $y \in K(A)$ is a uniformizer at $0$ along this arc,
then $y(k R)/y(R) \to k$ in $K_\pp$ (l'H\^opital's rule).
Let $a \in A(K)$ be a point of infinite order.
Using formal group coordinates shows that $N! a \in A(K_\pp)$ tends to $0$ along such an arc as $N \to \infty$,
and the function $y \colonequals \sum t_i y_i$ is a uniformizer at $0$ along the arc for any $(t_1,\ldots,t_g) \in K^g$ outside a hyperplane $H_\pp \subset \Aff^g_{K_\pp}$.
Now, given \emph{finitely many} primes $\pp_1,\ldots,\pp_m$ of $\calO_K$,
we can choose $(t_1,\ldots,t_g) \in T^g$ outside all of $H_{\pp_1},\ldots,H_{\pp_m}$, 
since $T$ is infinite;
then $y(k(N! a))/y(N! a) \to k$ in $K_{\pp_i}$ for each $i \in \{1,\ldots,m\}$.
Thus $S$ weakly approximates $\Z$.
\item
Let $n$ be as in Lemma~\ref{L:congruence}.
Let $U'$ be the set of $\alpha \in \calO_L$ such that
there exist $k \in S$ and $I = \num(s)$ for some $s \in S$
such that
$(\alpha-1)\cdots(\alpha-n) \mid I$ and $\alpha \equiv k \pmod{I}$.
The end of Section~\ref{S:diophantine} implies that $U'$ is $\calO_L$-diophantine.
By Lemma~\ref{L:congruence}, $U' \subset \calO_K$.

If $\alpha \in \Z - \{1,\ldots,n\}$,
Lemma~\ref{L:numerator} provides $s \in S$ such that
$(\alpha-1)\cdots(\alpha-n) \mid I \colonequals \num(s)$.
Since $S$ weakly approximates $\Z$,
there exists $k \in S$ such that $k \equiv \alpha \pmod{I}$.
Thus $\alpha \in U'$.

Take $U \colonequals U' \union \{1,\ldots,n\}$, which is $\calO_L$-diophantine.
\item
Let $b_1,\ldots,b_\kappa$ be a $\Z$-basis of $\calO_K$.
Then $\calO_K = \sum_{i=1}^\kappa U b_i$, which is $\calO_L$-diophantine.
\qedhere
\end{enumerate}
\end{proof}

\begin{remark}
In order to guarantee that some $y$ was a uniformizer along the arc for each of $\pp_1,\ldots,\pp_m$, we let $y$ range over all linear combinations of $y_1,\ldots,y_g$ with coefficients in an infinite set $T$.
But in fact, if we assume (as we may) that $A$ is simple, then already $y_1$ suffices, because the $p$-adic analogue of W\"ustholz's analytic subspace theorem implies that for every $\pp$, the $\pp$-adic logarithm $\log_\pp a \in \Lie A_{K_\pp}$ does not lie in any hyperplane defined over $K$; see \cite[Theorem~1]{Matev2010+} or \cite[Proposition~2.5]{Fuchs-Pham2015}.
\end{remark}

\section{Shlapentokh's reduction}
\label{S:Shlapentokh}

\begin{theorem}[Shlapentokh]
\label{T:Z is diophantine}
Assume that
\begin{itemize}
\item $\Z$ is $\calO_E$-diophantine for every totally real number field $E$, and
\item $\calO_K$ is $\calO_L$-diophantine for every degree~$2$ extension of number fields $L/K$.
\end{itemize}
Then $\Z$ is $\calO_F$-diophantine for every number field $F$.
\end{theorem}

\begin{proof}
If $F'/F$ is a finite extension and $\Z$ is $\calO_{F'}$-diophantine,
then $\Z$ is also $\calO_F$-diophantine.
Thus we may enlarge $F$ to assume that $F$ is Galois over $\Q$.

For each complex conjugation $\sigma \in \Aut F$ arising from a nonreal emebedding $F \injects \C$, we have $[F:F^\sigma]=2$, so $\calO_{F^\sigma}$ is $\calO_F$-diophantine.
Let $E = \Intersection_\sigma F^\sigma$.
Then the intersection $\calO_E = \Intersection_\sigma \calO_{F^\sigma}$ is $\calO_F$-diophantine.
On the other hand, $E$ is totally real, so $\Z$ is $\calO_E$-diophantine by assumption.
By transitivity, $\Z$ is $\calO_F$-diophantine.
\end{proof}

\begin{remark}
\cite[Corollary~2.5]{Koymans-Pagano2025+} proved $\calA_{K,L}$ (and hence that $\calO_K$ is $\calO_L$-diophantine) not for all degree~$2$ extensions $L/K$, but only those satisfying all of the following additional assumptions: $K \subset \R$, $L$ is Galois over $\Q$, and $L \supset L_0 \colonequals \Q(\sqrt{-1},\sqrt{5},\sqrt{7},\sqrt{11},\sqrt{13},\sqrt{17},\sqrt{19})$.
But it is easy to adapt the proof of Theorem~\ref{T:Z is diophantine} to use its second hypothesis only for these extensions, by enlarging $F$ to contain $L_0$.
\end{remark}




\bibliographystyle{amsalpha}
\bibliography{h10_over_OK}

\end{document}